\documentclass[12pt,oneside]{amsart}

\usepackage{amsmath}
\usepackage{amssymb}
\usepackage{amsthm}
\usepackage{mathtools}
\usepackage{multirow}
\usepackage{enumitem}
\usepackage[square,numbers,sort&compress]{natbib}
\usepackage{url}

\newtheorem{theorem}{Theorem}[section]
\newtheorem{lemma}[theorem]{Lemma}
\newtheorem{proposition}[theorem]{Proposition}

\newtheorem{question}[theorem]{Question}
 
\theoremstyle{definition}
\newtheorem{definition}[theorem]{Definition}

\theoremstyle{remark}
\newtheorem{remark}[theorem]{Remark}
\newtheorem{example}[theorem]{Example}
\newtheorem{case}{Case}

\numberwithin{subcase}{case}

\numberwithin{subcase2}{case2}

\newcommand{\overbar}[1]{\mkern 1.5mu\overline{\mkern-1.5mu#1\mkern-1.5mu}\mkern 1.5mu}

\makeatletter
\@addtoreset{case}{theorem}
\makeatother

\makeatletter
\def\@seccntformat#1{%
  \protect\textup{\protect\@secnumfont
    \ifnum\pdfstrcmp{subsection}{#1}=0 \bfseries\fi
    \csname the#1\endcsname
    \protect\@secnumpunct
  }%
}  
\makeatother

\makeatletter
\def\subsection{\@startsection{subsection}{3}%
  \z@{.7\linespacing\@plus.7\linespacing}{0.5\linespacing}%
  {\normalfont\bfseries}}
\makeatother


\begin{document}
\title[Strong $\boldsymbol{(2\cdot t)}$ and Strong $\boldsymbol{(3\cdot t)}$ Transformations]{Strong $\boldsymbol{(2\cdot t)}$ and Strong $\boldsymbol{(3\cdot t)}$ Transformations for Strong \textit{M}-equivalence}
\author{Ghajendran Poovanandran}
\address{School of Mathematical Sciences\\
Universiti Sains Malaysia\\
11800 USM, Malaysia}
\email{p.ghajendran@gmail.com}
\author{Wen Chean Teh}
\address{School of Mathematical Sciences\\
Universiti Sains Malaysia\\
11800 USM, Malaysia}
\email[Corresponding author]{dasmenteh@usm.my}

\begin{abstract}
Parikh matrices have been extensively investigated due to their usefulness in studying subword occurrences in words. Due to the dependency of Parikh matrices on the ordering of the alphabet, strong $M$\!-equivalence was proposed as an order-independent alternative to $M$\!-equivalence in studying words possessing the same Parikh matrix. This paper introduces and studies the notions of strong $(2\cdot t)$ and strong $(3\cdot t)$ transformations in determining when two ternary words are strongly $M$\!-equivalent. The irreducibility of strong $(2\cdot t)$ transformations are then scrutinized, exemplified by a structural characterization of irreducible strong $(2\cdot 2)$ transformations. The common limitation of these transformations in characterizing strong $M$\!-equivalence is then addressed.
\end{abstract}

\smallskip

\maketitle
\section{Introduction}
The extension of Parikh vectors \cite{rP66} into Parikh matrices \cite{MSSY01} is known for its usefulness in studying (scattered) subword occurrences in words (for example, see \cite{MSY04,aS05b,aS06}). However, not every word is uniquely determined by its Parikh matrix. The injectivity problem, which asks for a  characterization of words sharing the same Parikh matrix, has received extensive interest (for example, see \cite{aA07,AAP08,aA14,AMM02,AT16,FR04,MS12,aS10,vS09,SS06,wT14,wT16,wT15b,wT16b,wT16c,TK15}). This, together with the fact that Parikh matrices are dependent on the ordering of the alphabet, led to the introduction of strong $M$\!-equivalence in \cite{wT15b}. Two words are strongly $M$\!-equivalent if and only if they share the same Parikh matrix with respect to every ordered alphabet with the same underlying alphabet.

The characterization of strongly $M$\!-equivalent words, which we now term as the strong injectivity problem, remains open even for the case of the ternary alphabet. In this paper, we develop a canonical generalization (exclusively for the ternary alphabet) of certain elementary rewriting rules that preserves strong $M$\!-equivalence, which we propose as the strong $(2\cdot t)$ transformation\footnote{The choice of term is due to the fact that this transformation is an analogue of the $(2\cdot t)$ transformation, introduced by Atanasiu in \cite{aA14}.}. In addition, we introduce new symmetric transformations that preserve strong \mbox{$M$\!-equivalence} of ternary words, which we term as the strong $(3\cdot t)$ transformations.

The remainder of this paper is structured as follows. Section 2 provides the basic definitions and terminology. The next section introduces and studies the notions of strong $(2\cdot t)$ and strong $(3\cdot t)$ transformations. Following up, Section~4 looks into the reducibility of strong $(2\cdot t)$ transformations into simpler ones, where the irreducible transformations are characterized. Section 5 addresses the extent of strong $(2\cdot t)$ and strong $(3\cdot t)$ transformations in characterizing strong $M$\!-equivalence for the ternary alphabet. Finally, our conclusion follows after that.

\section{Preliminaries}
The cardinality of a set $A$ is denoted by $|A|$.

Let $\Sigma$ be an (finite) alphabet. The set of all words over $\Sigma$ is denoted by $\Sigma^*$. The unique empty word is denoted by $\lambda$. Given two words $v,w\in\Sigma^*$, the concatenation of $v$ and $w$ is denoted by $vw$. In this paper, we frequently deal with alphabets with a total ordering assigned to it (i.e. \textit{ordered alphabets}). For example, if $a_1<a_2<\cdots <a_k$, then we may write $\Sigma=\{a_1<a_2<\cdots <a_k\}$. Conversely, if $\Sigma=\{a_1<a_2<\cdots <a_k\}$ is an ordered alphabet, we shall term $\{a_1,a_2,\ldots ,a_k\}$ as the \textit{underlying alphabet}. For convenience, we shall abuse notation and use $\Sigma$ to denote both the ordered alphabet and its underlying alphabet. For $\Sigma=\{a_1<a_2<\cdots <a_s\}$, we denote by $a_{i,j}$ the word $a_ia_{i+1}\ldots a_j$ for $1\le i\le j\le s$.

\begin{definition}
Suppose $\Sigma$ is an alphabet and $v,w\in\Sigma^*$. 
\begin{enumerate}[leftmargin=1.1cm]
\item We say that $v$ is a \textit{scattered subword} (or simply \textit{subword}) of $w$ if and only if there exist $x_1,x_2,\ldots ,x_n,y_0,y_1,\ldots ,y_n\in\Sigma^*$ with some of them possibly being empty, such that
\begin{center}
$v=x_1x_2\cdots x_n$ and $w=y_0x_1y_1\cdots y_{n-1}x_ny_n$.
\end{center}
\item We say that $v$ is a \textit{factor} of $w$ if and only if there exist $x,y\in\Sigma^*$ such that $w=xvy$.
\end{enumerate}
\end{definition}

We denote by $|w|_v$, the number of occurrences of the word $v$ as a subword of $w$. If two occurrences of subword $v$ in a word $w$ differ by at least one position of any letter, then the two occurrences are considered different. For example, $|abccc|_{abc}=3$. By convention, $|w|_\lambda=1$ for all $w\in\Sigma^*$.

For $k\ge 1$, let $\mathcal{M}_k$ denote the multiplicative monoid of $k\times k$ upper triangular matrices with nonnegative integral entries and unit diagonal.
\begin{definition}
Suppose $\Sigma=\{a_1<a_2<\cdots <a_k\}$ is an ordered alphabet. The \textit{Parikh matrix mapping} with respect to $\Sigma$, denoted by $\Psi_\Sigma$, is the morphism:
\begin{equation*}
\Psi_\Sigma:\Sigma^*\rightarrow\mathcal{M}_{k+1},
\end{equation*}
defined such that $\Psi_\Sigma(a_q)=(m_{i,j})_{1\le i,j\le k+1}$, where  
$m_{i,i}=1$ for all $1\le i\le k+1$, $m_{q,q+1}=1$, and all other entries of the matrix $\Psi_\Sigma(a_q)$ are zero. Matrices of the form $\Psi_\Sigma(w)$ for $w\in\Sigma^*$ are termed as  \textit{Parikh matrices}.
\end{definition}
\begin{theorem}\cite{MSSY01}\label{PropertiesParikhMat}
Suppose $\Sigma=\{a_1<a_2<\cdots <a_k\}$ is an ordered alphabet and $w\in\Sigma^*$. The matrix $\Psi_\Sigma(w)=(m_{i,j})_{1\le i,j\le k+1}$ has the following properties:
\begin{enumerate}
\item $m_{i,j}=0$ for all $1\le j<i\le k+1$;
\item $m_{i,i}=1$ for all $1\le i\le k+1$;
\item $m_{i,j+1}=|w|_{a_{i,j}}$ for $1\le i\le j\le k$.
\end{enumerate}
\end{theorem}
The second diagonal of the Parikh matrix $\Psi_\Sigma(w)$ of a word  $w\in\Sigma^*$ contains the \textit{Parikh vector} $\Psi(w)=(|w|_{a_1},|w|_{a_2},\ldots ,|w|_{a_k})$ of that word $w$.
\begin{example}\label{exampleParikhmap}
Consider the ordered alphabet $\Sigma=\{a<b<c\}$. Then, the Parikh matrix of the word $abccc$ with respect to $\Sigma$ can be computed as follows:
\begin{align*}
\Psi_\Sigma(abccc)&=\Psi_\Sigma(a)\Psi_\Sigma(b)\Psi_\Sigma(c)\Psi_\Sigma(c)\Psi_\Sigma(c)\\[2pt]
&=\begin{pmatrix}
1 & 1 & 0 & 0\\
0 & 1 & 0 & 0\\
0 & 0 & 1 & 0\\
0 & 0 & 0 & 1
\end{pmatrix} 
\begin{pmatrix}
1 & 0 & 0 & 0\\
0 & 1 & 1 & 0\\
0 & 0 & 1 & 0\\
0 & 0 & 0 & 1
\end{pmatrix}\cdots
\begin{pmatrix}
1 & 0 & 0 & 0\\
0 & 1 & 0 & 0\\
0 & 0 & 1 & 1\\
0 & 0 & 0 & 1
\end{pmatrix}\\
&=\begin{pmatrix}
1 & 1 & 1 & 3\\
0 & 1 & 1 & 3\\
0 & 0 & 1 & 3\\
0 & 0 & 0 & 1
\end{pmatrix}.
\end{align*}
\end{example}

\begin{definition}\label{DefMequivalence}
Suppose $\Sigma=\{a_1<a_2<\ldots <a_k\}$ is an ordered alphabet. Two words $w,w'\in\Sigma^*$ are {\it $M$\!-equivalent}, denoted by $w\equiv_Mw'$, if and only if $\Psi_\Sigma(w)=\Psi_\Sigma(w')$.
\end{definition}
For the ternary alphabet, the following two rewriting rules can be used to determine whether two words are $M$\!-equivalent. The first rule is elementary while the second rule was introduced by Atanasiu in \cite{aA14}. 

Suppose $\Sigma=\{a<b<c\}$ and $w,w'\in\Sigma^*$.

\vspace{0.5em}\hspace{-1.5em}\begingroup
\setlength{\tabcolsep}{2pt}
\begin{tabular}{lll}
Rule $E1$ &: &If $w=xacy$ and $w'=xcay$ for some $x,y\in\Sigma^*$, then $w\equiv_Mw'$.\\
Rule $E2\cdot t$ &: &Suppose $w$ contains $t\ge 1$ factors of the form
\end{tabular}
\endgroup\\
\begin{center}
$\beta_k =\begin{cases}
\alpha_kb x_k b\alpha_k &\text{for } 1\leq k \leq t'\\
b\alpha_k x_k \alpha_kb &\text{for } t'+1\leq k\leq t,
\end{cases}$
\end{center}
such that the following holds:
\begin{itemize}
\item $\alpha_k\in\{a,c\}$ and $x_k\in\Sigma^*$ for every $1\le k\le t$;
\item $\alpha_m b$, $b\alpha_m$, $\alpha_nb$ and $b\alpha_n$ do not overlap each other whenever $m\neq n$;
\item $\sum_{k=1}^{t'} |x_k|_{\overbar{\alpha_k}}=\sum_{k=t'+1}^t |x_k|_{\overbar{\alpha_k}}$, where $\overbar{\alpha_k}\in\Sigma\backslash\{\alpha_k, b\}$ for every $1\le k\le t$.
\end{itemize}
If $w'$ is obtained from $w$ by (simultaneously) swapping $\alpha_k b$ with $b\alpha_k$ for every $1\le k\le t$, then $w\equiv_M w'$.
\begin{remark}
The notion of \textit{elementary matrix equivalence (ME-equivalence)}, introduced in \cite{aS10}, is composed of Rule $E1$ and Rule $E2\cdot 1$.
\end{remark}

\begin{definition}\label{2tTransformation}
Suppose $\Sigma=\{a<b<c\}$ and $w,w'\in\Sigma^*$. We say that the transformation of $w$ into $w'$ is a \emph{($2\cdot t$) transformation}, denoted by $w\underset{(2\cdot t)}{\longrightarrow}w'$, if and only if $w'$ is obtained from $w$ by an application of Rule $E2\cdot t$.
\end{definition}

\begin{example}
Suppose $\Sigma=\{a<b<c\}$ and
\begin{equation*}
w=\boldsymbol{ab}c\boldsymbol{ba}\textbf{bc}a\textbf{cb}cb\underset{(2\cdot 2)}{\longrightarrow}bacabcba\boldsymbol{bccb}\underset{(2\cdot 1)}{\longrightarrow}bacabcbacbbc=w'.
\end{equation*}Then, $w\equiv_Mw'$.
\end{example}

\begin{definition}
Suppose $\Sigma$ is an alphabet and $w,w'\in\Sigma^*$. Then $w$ and $w'$ are \textit{strongly $M$\!-equivalent}, denoted by $w\overset{s}{\equiv}_Mw'$, if and only if $w\equiv_M w'$ with respect to any ordered alphabet with \mbox{underlying} alphabet $\Sigma$.
\end{definition}

\begin{remark}\label{SufficientOrder}\cite{GT16a}
Suppose $\Sigma=\{a,b,c\}$ and $w,w'\in\Sigma^*$. Then $w\overset{s}{\equiv}_Mw'$ if and only if $w\equiv_Mw'$ with respect to each of the ordered alphabets $\{a<b<c\}$, $\{b<a<c\}$, and $\{a<c<b\}$.
\end{remark}

The following Rule \emph{SE}, introduced in \cite{wT15b}, is elementary in deciding whether two words are strongly $M$\!-equivalent.
\vspace{0.5em}\begin{itemize}[leftmargin=1.8cm]
\item[Rule \textit{SE}.] Suppose $\Sigma$ is an alphabet and $w,w'\in\Sigma^*$. If $w=xabybaz$ and $w'=xbayabz$ for some distinct $a,b\in\Sigma$, $x,z\in\Sigma^*$, and $y\in\{a,b\}^*$, then $w\overset{s}{\equiv}_Mw'$.
\end{itemize}

\begin{definition}
Suppose $\Sigma=\{a,b,c\}$. Two words $w,w'\in\Sigma^*$ are {\it strongly elementarily M-equivalent (MSE-equivalent)}, denoted by $w\equiv_{\text{\textit{MSE}}}w'$, if and only if $w'$ can be obtained from $w$ by finitely many applications of Rule \textit{SE}.
\end{definition}

\section{Strong ($2\cdot t$) and Strong ($3\cdot t$) Transformations}\label{strictsec}

\subsection{Strong $\boldsymbol{(2\cdot t)}$ Transformation}

\begin{theorem}\label{Strong2tTransformation}
Suppose $\Sigma=\{a,b,c\}$ and $w,w'\in\Sigma^*$. Suppose $w$ contains $t\ge 1$ factors of the form
\begin{equation*}
\beta_k =\begin{cases}
ab x_k ba &\text{for } 1\leq k \leq t'\\
ba x_k ab &\text{for } t'+1\leq k\leq t
\end{cases}
\end{equation*}\\
such that the following conditions hold:
\begin{enumerate}
\item $x_k\in\Sigma^*$ for $1\le k\le t$;
\item ab and ba in $\beta_m$ do not overlap with ab and ba in $\beta_n$ whenever $m\neq n$;
\item $\sum_{k=1}^{t'} |x_k|_c=\sum_{k=t'+1}^t |x_k|_c$.
\end{enumerate}
If $w'$ is obtained from $w$ by (simultaneously) rewriting $\beta_k$ into $\beta_k'$ for every $1\le k\le t$, where
\begin{equation*}
\beta_k =\begin{cases}
ba x_k ab &\text{for } 1\leq k \leq t'\\
ab x_k ba &\text{for } t'+1\leq k\leq t,
\end{cases}
\end{equation*}\\
then $w$ and $w'$ are strongly $M$\!-equivalent.
\end{theorem}
\begin{proof}By Remark \ref{SufficientOrder}, it is sufficient to prove $w\equiv_Mw'$ with respect to each of the following ordered alphabets:
\begin{case}
$\{a<b<c\}$.\\
The transformation of $w$ into $w'$ is a ($2\cdot t$) transformation under $\{a<b<c\}$. Thus, $w\equiv_M w'$.
\end{case}
\begin{case}
$\{b<a<c\}$.\\
The transformation of $w$ into $w'$ is a ($2\cdot t$) transformation under $\{b<a<c\}$. Thus, $w\equiv_M w'$.
\end{case}
\begin{case}
$\{a<c<b\}$.\\
Since $a$ and $b$ are not consecutive, the swapping of $ab$ and $ba$ is simply by Rule $E1$. Thus, $w\equiv_M w'$.
\end{case}
Hence, $w\overset{s}{\equiv}_Mw'$.
\end{proof}
\begin{definition}\label{DefineStrong2t}
Suppose $\Sigma=\{a,b,c\}$ and $w$,$w'\in {\Sigma}^*$. If $w'$ is obtained from $w$ as in Theorem \ref{Strong2tTransformation}, where the number of factors $\beta_k$ involved is $t$, we say that the transformation of $w$ into $w'$ is a \emph{strong ($2\cdot t$) transformation}, denoted by $w\overset{s}{\underset{(2\cdot t)}{\longrightarrow}}w'$. The rewriting rule used on $w$ to obtain $w'$ is addressed as Rule $S2\cdot t$.
\end{definition}

\begin{remark}\label{SEBaseCase}
Rule $S2\cdot 1$ is actually Rule \textit{SE}, that is to say, Rule \textit{SE} is simply the basic case of Rule $S2\cdot t$. 
\end{remark}

\begin{example}\label{histexam}
$w=b\textit{\textbf{ab}}c\textit{\textbf{ba}}bc\textit{\textbf{ba}}bcb\textit{\textbf{ab}}\overset{s}{\underset{(2\cdot 2)}{\longrightarrow}}b\textit{\textbf{ba}}c\textit{\textbf{ab}}bc\textit{\textbf{ab}}bcb\textit{\textbf{ba}}=w'$.
\end{example}
Example \ref{histexam} holds the pair of words used as a counterexample by {\c S}erb{\v a}nu{\c t}{\v a} in \cite{vS09} to show that \textit{ME}-equivalence does not characterize \textit{M}-equivalence. They are strongly $M$\!-equivalent with respect to the alphabet $\{a,b,c\}$ as well. However, since Rule \textit{SE} cannot be applied anywhere on the first word, they are not \textit{MSE}-equivalent.

\begin{remark}\label{FactorsNotUnique}
The set of $t$ factors associated with a strong $(2\cdot t)$ transformation is not invariably unique.
\end{remark}
\begin{example}
Consider $\boldsymbol{ac}b\boldsymbol{ca}b\boldsymbol{ca}b\boldsymbol{ac}\overset{s}{\underset{(2\cdot 2)}{\longrightarrow}}cabacbacbca$. Then, the associated factors $\{\beta_1,\beta_2\}$ can be either $\{\boldsymbol{ac}b\boldsymbol{ca},\boldsymbol{ca}b\boldsymbol{ac}\}$ or $\{\boldsymbol{ac}bcab\boldsymbol{ca},\boldsymbol{ca}bcab\boldsymbol{ac}\}$.
\end{example}
Remark \ref{SEBaseCase} highlights the fact that Rule $S2\cdot t$ provides a richer partial characterization of strong M-equivalence for the ternary alphabet, compared to Rule \textit{SE}. At this point, we should attempt to answer the following question. 
\begin{question}
For which subclass of ternary words does strong ($2\cdot t$) transformation characterize strong M-equivalence but \textit{MSE}-equivalence fails to do so?
\end{question}
The answer to this question remains as an open problem. However, the next theorem attempts to provide a partial insight to this question, where pairs of words that are strongly $M$\!-equivalent under the strong $(2\cdot t)$ transformation but not \textit{MSE}-equivalent are generated. 

\begin{lemma}\label{ClosedUnderSE}
Suppose $\Sigma=\{a,b,c\}$, $w,w'\in\Sigma^*$ with $|w|_c\ge 1$. Suppose  $w=u_0cu_1c\cdots cu_t$ for some $u_i\in\{a,b\}^*$ $(0\le i\le t)$ such that $|u_i|_a,|u_i|_b\ge 1$ for every $1\le i\le t-1$. Write $w$ in the form $w_1cw_2$ for some $w_1,w_2\in\Sigma^*$. If $w'\equiv_{\textit{MSE}}w$, then $w'$ has the form $w'_1cw'_2$ for some $w'_1,w'_2\in\Sigma^*$ such that $w'_1\equiv_{\textit{MSE}}w_1$ and $w'_2\equiv_{\textit{MSE}}w_2$.
\end{lemma}
\begin{proof}
The proof is straightforward by induction. Note that for the base step, since $|u_i|_a,|u_i|_b\ge 1$ for every $1\le i\le t-1$ and there is only one $c$ between every pair of factors $(u_i,u_{i+1})$, no application of Rule \textit{SE} involving any of the character $c$ can be applied on $w$. Hence, Rule \textit{SE} can only be applied on any of the factors $u_i$. Let $w''=u''_0cu''_1c\cdots cu''_t$ be the word obtained by one application of Rule~\textit{SE} on $w$. It is clear that $u_i\equiv_\textit{MSE}u''_i$ for every $0\le i\le t$, thus the base step holds by some simple further reasoning. The induction step easily follows by the induction hypothesis and transitivity of \textit{MSE}-equivalence.
\end{proof}

\begin{theorem}\label{StrongMButNotMSE}
Suppose $\Sigma=\{a,b,c\}$ and $w,w'\in\Sigma^*$. Let $w=w_0\beta_1w_1\beta_2\cdots \beta_tw_t$ for some $t\ge 1$ such that $w_i\in\Sigma^*$ for every $0\le i\le t$ and
\vspace{-0.3em}\begin{equation*}
\beta_k =\begin{cases}
ab x_k ba &\text{for } 1\leq k \leq t'\\
ba x_k ab &\text{for } t'+1\leq k\leq t,
\end{cases}
\end{equation*} satisfying conditions as in Theorem \ref{Strong2tTransformation}. In addition, suppose that $|x_k|_c\le 1$ for every $1\le k\le t$, $|x_k|_c=1$ for some $1\le k\le t$, and $|w_i|_c\le 1$ for every $0\le i\le t$.

Let $w'$ be the word obtained by (simultaneously) swapping every $ab$ with $ba$ in $w$. Then $w\overset{s}{\underset{(2\cdot t)}{\longrightarrow}}w'$ but $w$ is not \textit{MSE}-equivalent to $w'$.
\end{theorem}
\begin{proof}

It suffices to show that $w$ and $w'$ are not \textit{MSE}-equivalent since $w\overset{s}{\underset{(2\cdot t)}{\longrightarrow}}w'$ by Definition \ref{DefineStrong2t}. 

We argue by contradiction. Assume $w$ and $w'$ are \textit{MSE}-equivalent. Since $|x_k|_c\le 1$ for every $1\le k\le t$, $|x_k|_c=1$ for some $1\le k\le t$, and $|w_i|_c\le 1$ for every $0\le i\le t$, it follows that $w$ is in the form stated in Lemma \ref{ClosedUnderSE}. Now, choose \mbox{$1\le k_0\le t$} such that $|x_{k_0}|_c=1$. Without loss of generality, assume $\beta_{k_0}=ab x_{k_0} ba$. We write $w$ in the form of  $u\boldsymbol{ab}\underbrace{x_{{k_0},1}\boldsymbol{c}x_{{k_0},2}}_{x_{k_0}}\boldsymbol{ba}v$ and $w'$ in the form of $u'\boldsymbol{ba}\underbrace{x_{{k_0},1}\boldsymbol{c}x_{{k_0},2}}_{x_{k_0}}\boldsymbol{ab}v'$, where $u,u',v,v'\in\Sigma^*$. Then, by Lemma~\ref{ClosedUnderSE}, $u\boldsymbol{ab}x_{{k_0},1}\equiv_{\textit{MSE}}u'\boldsymbol{ba}x_{{k_0},1}$ and $x_{{k_0},2}\boldsymbol{ba}v\equiv_{\textit{MSE}}x_{{k_0},2}\boldsymbol{ab}v'$.

Since the rewriting of each factor $\beta_k$ preserves the number of occurrences of $ab$ in $w$, it follows that $|u|_{ab}=|u'|_{ab}$. 
Observe that  $|u\boldsymbol{ab}x_{{k_0},1}|_{ab}=|u'\boldsymbol{ba}x_{{k_0},1}|_{ab}+1$, thus $\Psi_{\{a<b<c\}}(u\boldsymbol{ab}x_{{k_0},1})\neq\Psi_{\{a<b<c\}}(u'\boldsymbol{ba}x_{{k_0},1})$. It immediately follows that $u\boldsymbol{ab}x_{{k_0},1}$ and  $u'\boldsymbol{ba}x_{{k_0},1}$ are not strongly $M$\!-equivalent, and thereby they cannot be \textit{MSE}-equivalent. This is a contradiction, therefore we conclude that $w$ and $w'$ are not \textit{MSE}-equivalent.
\end{proof}

\subsection{Strong $\boldsymbol{(3\cdot t)}$ Transformation}

Although Rule $S2\cdot t$ is a generalization of Rule \textit{SE} for the ternary alphabet, it is still far from characterizing every pair of strongly $M$\!-equivalent ternary words (see Proposition \ref{ExampleStrongButNot2t}). In this subsection, we develop another canonical rewriting rule that preserves strong $M$\!-equivalence for the ternary alphabet.

\begin{lemma}\label{LemmaChangeInTopRight}\cite{GT16a}
Suppose $\Sigma=\{a,b,c\}$ and $w,w'\in\Sigma^*$. Assume that $w=xabybaz$ and $w'=xbayabz$ for some $x,y,z\in\Sigma^*$, then the following statements are true:
\begin{enumerate}
\item $|w|_{abc}-|w'|_{abc}=|w|_{cba}-|w'|_{cba}=|y|_c.$
\item $|w|_{bac}-|w'|_{bac}=|w|_{cab}-|w'|_{cab}=-|y|_c.$   
\item $|w|_{acb}-|w'|_{acb}=|w|_{bca}-|w'|_{bca}=0.$
\end{enumerate}
\end{lemma}

\begin{theorem}\label{AnotherStrongM}
Suppose $\Sigma=\{a,b,c\}$ and $w,w'\in\Sigma^*$. Suppose $w$ contains $t\ge 1$ factors of the form
\begin{equation*}
\beta_k =\begin{cases}
ab x_k ba &\text{for } 1\leq k \leq t'\\
bc x_k cb &\text{for } t'+1\leq k\leq t''\\
ca x_k ac &\text{for } t''+1\leq k\leq t\\
\end{cases}
\end{equation*}\\
such that the following conditions hold:
\begin{enumerate}
\item $x_k\in\Sigma^*$ for $1\le k\le t$;
\item The first and last two characters in $\beta_m$ do not overlap with the first and last two characters in $\beta_n$ whenever $m\neq n$;
\item $\sum_{k=1}^{t'} |x_k|_c=\sum_{k=t'+1}^{t''} |x_k|_a=\sum_{k=t''+1}^{t} |x_k|_b$.
\end{enumerate}
If $w'$ is obtained from $w$ by (simultaneously) rewriting $\beta_k$ into $\beta_k'$ for every $1\le k\le t$, where
\begin{equation*}
\beta_k' =\begin{cases}
ba x_k ab &\text{for } 1\leq k \leq t'\\
cb x_k bc &\text{for } t'+1\leq k\leq t''\\
ac x_k ca &\text{for } t''+1\leq k\leq t\\
\end{cases}
\end{equation*}\\
then $w$ and $w'$ are strongly $M$\!-equivalent.
\end{theorem}
\begin{proof}
Let $u\in\Sigma^*$ be such that $|u|_a\le 1$ for all $a\in\Sigma$ and $|u|\le 2$. Then, it is obvious that $|w|_u$=$|w'|_u$. By Remark \ref{SufficientOrder}, it remains to be shown that $|w|_v=|w'|_v$ for every $v\in\{abc,bac,acb\}$.

Suppose $w=w_0\rightarrow w_1\rightarrow\cdots w_t=w'$. Let $w_k$ be the word obtained from $w$ by rewriting each $\beta_i$ into $\beta_i'$ for $1\le i\le k$. For each $1\le k\le t$, let $\alpha_k$ be the factor in $w_{k-1}$ that occupies the same position as the one occupied by $\beta_k$ in $w$. Then,
\begin{equation*}
\alpha_k =\begin{cases}
ab y_k ba &\text{for } 1\leq k \leq t'\\
bc y_k cb &\text{for } t'+1\leq k\leq t''\\
ca y_k ac &\text{for } t''+1\leq k\leq t\\
\end{cases}
\end{equation*}\\
for some $y_k\in\Sigma^*$ such that $\Psi(y_k)=\Psi(x_k)$ for all $1\le k\le t$. Now, let
\begin{equation*}
\alpha_k' =\begin{cases}
ba y_k ab &\text{for } 1\leq k \leq t'\\
cb y_k bc &\text{for } t'+1\leq k\leq t''\\
ac y_k ca &\text{for } t''+1\leq k\leq t.\\
\end{cases}
\end{equation*}\\
Then, $w_k$ is the word obtained from $w_{k-1}$ by rewriting $\alpha_k$ into $\alpha_k'$. Let $\Omega_k$ be the triplet $(|w_k|_{abc}-|w_{k-1}|_{abc},|w_k|_{bac}-|w_{k-1}|_{bac},|w_k|_{acb}-|w_{k-1}|_{acb})$. Then, $\sum\limits_{k=1}^{t}\Omega_k=(|w'|_{abc}-|w|_{abc},|w'|_{bac}-|w|_{bac},|w'|_{acb}-|w|_{acb})$. By Lemma~\ref{LemmaChangeInTopRight}, it holds that
\begin{equation*}
\Omega_k =\begin{cases}
(|y_k|_c,-|y_k|_c,0) &\text{for } 1\leq k \leq t'\\
(-|y_k|_a,0,|y_k|_a) &\text{for } t'+1\leq k\leq t''\\
(0,|y_k|_b,-|y_k|_b) &\text{for } t''+1\leq k\leq t.\\
\end{cases}
\end{equation*}\\
Note that $\sum\limits_{k=1}^{t}\Omega_k$
\begin{equation*}
\begin{split}
={}&\left(\sum\limits_{k=1}^{t'}|y_k|_c-\sum\limits_{k=t'+1}^{t''}|y_k|_a,\sum\limits_{k=t''+1}^{t}|y_k|_b-\sum\limits_{k=1}^{t'}|y_k|_c,\sum\limits_{k=t'+1}^{t''}|y_k|_a-\sum\limits_{k=t''+1}^{t}|y_k|_b\right)\\
=&{}\left(\sum\limits_{k=1}^{t'}|x_k|_c-\sum\limits_{k=t'+1}^{t''}|x_k|_a,\sum\limits_{k=t''+1}^{t}|x_k|_b-\sum\limits_{k=1}^{t'}|x_k|_c,\sum\limits_{k=t'+1}^{t''}|x_k|_a-\sum\limits_{k=t''+1}^{t}|x_k|_b\right)
\end{split}
\end{equation*}
because $\Psi(y_k)=\Psi(x_k)$ for all $1\le k\le t$.

By condition (3), $\sum\limits_{k=1}^{t}\Omega_k=(0,0,0)$. Thus $|w|_{abc}=|w'|_{abc}$, $|w|_{bac}=|w'|_{bac}$, and $|w|_{acb}=|w'|_{acb}$, concluding that $w$ and $w'$ are strongly $M$\!-equivalent.
\end{proof}

\begin{definition}
Suppose $\Sigma=\{a,b,c\}$, $w,w'\in {\Sigma}^*$, and $t\ge 1$. We say that the transformation of $w$ into $w'$ is a \textit{strong $(3\cdot t)$ transformation}, denoted by $w\overset{s}{\underset{(3\cdot t)}{\longrightarrow}}w'$, if and only if $w'$ is obtained from $w$ as in Theorem \ref{AnotherStrongM} and the number of factors $\beta_k$ involved is $t$. The rewriting rule used on $w$ to obtain $w'$ is addressed as Rule $S3\cdot t$.
\end{definition}

\subsection{Mutually Exclusivity of Strong $\boldsymbol{(2\cdot t)}$ and Strong $\boldsymbol{(3\cdot t)}$ Transformations}

The following propositions should put to rest any doubts on why Rule $S2\cdot t$ and Rule $S3\cdot t$ are both needed in the aim to  characterize strong $M$\!-equivalence for the ternary alphabet. 

\begin{proposition}\label{ExampleStrongButNot3t}
Suppose $\Sigma=\{a,b,c\}$.  Then, there exist infinitely many pairs of words $w$ and $w'$ in $\Sigma^*$ such that $w'$ is obtained from $w$ by some strong $(2\cdot t)$ transformation but cannot by finitely many strong $(3\cdot t)$ transformations.
\end{proposition}
\begin{proof}
Fix an integer $t\ge 1$. Consider the words $w=(abcbabacab)^t$ and $w'=(bacababcba)^t$. Then,
\begin{equation*}
w=(\boldsymbol{ab}c\boldsymbol{ba}\textbf{ba}c\textbf{ab})^t\overset{s}{\underset{(2\cdot 2t)}{\longrightarrow}}(\boldsymbol{ba}c\boldsymbol{ab}\textbf{ab}c\textbf{ba})^t=w'. 
\end{equation*}
However, observe that for any positive integer $t'$, Rule $S3\cdot t'$ cannot be applied anywhere on $w$. Thus, the conclusion holds.
\end{proof}

\begin{proposition}\label{ExampleStrongButNot2t}
Suppose $\Sigma=\{a,b,c\}$.  Then, there exist infinitely many pairs of words $w$ and $w'$ in $\Sigma^*$ such that $w'$ is obtained from $w$ by some strong $(3\cdot t)$ transformation but cannot by finitely many strong $(2\cdot t)$ transformations.
\end{proposition}
\begin{proof}
Fix an integer $m\ge 1$. Consider the words $w=a^mbcbabcacbcabac^m$ and $w'=a^{m-1}bacabcbabcacbcac^{m-1}$. Then,
\begin{equation*}
w=a^{m-1}\boldsymbol{ab}c\boldsymbol{ba}\textbf{bc}a\textbf{cb}\boldsymbol{ca}b\boldsymbol{ac}c^{m-1}\overset{s}{\underset{(3\cdot 3)}{\longrightarrow}}a^{m-1}bacabcbabcacbcac^{m-1}=w'.
\end{equation*}
However, observe that for any positive integer $t$, Rule $S2\cdot t$ cannot be applied anywhere on $w$. Thus, the conclusion holds.
\end{proof}

\section{Irreducible Strong ($2\cdot t$) Transformation}
In this section, we study strong $(2\cdot t)$ transformations in terms of their  decomposability into simpler ones.

\begin{definition}
Suppose $\Sigma=\{a,b,c\}$ and $w,w'\in\Sigma^*$, where $w\overset{s}{\underset{(2\cdot t)}{\longrightarrow}}w'$. We say that the strong ($2\cdot t$) transformation of $w$ into $w'$ is \textit{reducible} if and only if there exists $w''\in\Sigma^*$ such that $w\overset{s}{\underset{(2\cdot t_1)}{\longrightarrow}}w''\overset{s}{\underset{(2\cdot t_2)}{\longrightarrow}}w'$ for some $t_1,t_2>0$ with $t_1+t_2=t$. Otherwise, we say that the transformation is \textit{irreducible}.
\end{definition}

The following shows that every strong ($2\cdot t$) transformation is made up of irreducible ones. One can easily prove it by induction.
\begin{remark}
Suppose $\Sigma=\{a,b,c\}$ and $w,w'\in\Sigma^*$, where $w\overset{s}{\underset{(2\cdot t)}{\longrightarrow}}w'$. Then, there exist $w_i\in\Sigma^*,0\le i\le n$ such that $w=w_0\overset{s}{\underset{(2\cdot t_1)}{\longrightarrow}}w_1\overset{s}{\underset{(2\cdot t_2)}{\longrightarrow}}\cdots \overset{s}{\underset{(2\cdot t_n)}{\longrightarrow}}w_n=w'$, where $w_{i-1}\overset{s}{\underset{(2\cdot t)}{\longrightarrow}}w_i$ is irreducible for each $1\le i\le n$ and $\sum\limits_{i=1}^{n}t_i=t$.
\end{remark}

Next, we develop a characterization of irreducible strong ($2\cdot t$) transformations. We first establish a factor-independent description of strong ($2\cdot t$) transformations due to Remark \ref{FactorsNotUnique}.

Suppose $\Sigma$ is a ternary alphabet and $w,w'\in\Sigma^*$, where  $w\overset{s}{\underset{(2\cdot t)}{\longrightarrow}}w'$. Then $w'$ is obtained from $w$ by rewriting $2t$ non-overlapping pairs of characters as in Theorem \ref{Strong2tTransformation}. Let $\mu_i,1\le i\le 2t$ enumerate, from left to right, each of those pair of characters in $w$. Let $\mu_1$ be $ab$ for some $a,b\in\Sigma$. Then, for all $1\le i\le 2t$, $\mu_i$ is either $ab$ or $ba$. For each occurrence of $\mu_i$, assuming that $w=x_i\mu_i y_i$ for some $x_i,y_i\in\Sigma^*$, we associate to it an ordered pair $(p_i,q_i)$, defined as the following:
\begin{equation}\tag{$\ast$}
p_i=\begin{cases}
-1 &\text{if }\mu_i=ab\\
1 &\text{if }\mu_i=ba,
\end{cases}\,\hspace{1.5em}
q_i=\begin{cases}
-|y_i|_c &\text{if }\mu_i=ab\\
|y_i|_c &\text{if }\mu_i=ba,
\end{cases}
\end{equation}
where $c\in\Sigma\backslash\{a,b\}$. Note that $p_i$ (respectively $q_i$) accounts for the change in the number of occurrences of $ab$ (respectively $abc$) in $w$ when the characters in $\mu_i$ are interchanged.

Note that the definition of $(p_i,q_i)$ in $(\ast)$ can be extended to the case whenever one word is obtained from another by rewriting $2t$ non-overlapping pairs of characters, where for some $a,b\in\Sigma$, each pair is either $ab$ or $ba$. We will use this definition of $(p_i,q_i)$ in Theorem \ref{Characterize2t}.
\begin{example}
Consider $w=\boldsymbol{ac}b\boldsymbol{ca}\boldsymbol{ca}b\boldsymbol{ac}\overset{s}{\underset{(2\cdot 2)}{\longrightarrow}}\boldsymbol{ca}b\boldsymbol{ac}\boldsymbol{ac}b\boldsymbol{ca}=w'$. Then, enumerating from left to right, we have $\mu_1=(-1,-2)$, $\mu_2=(1,1)$, $\mu_3=(1,1)$, and $\mu_4=(-1,0)$.
\end{example}

\begin{remark}\label{RelateStrong2tAnd2t}
Suppose $\Sigma$ is a ternary alphabet, $w,w'\in\Sigma^*$ and $t\ge 1$. Suppose $w'$ is obtained from $w$ by rewriting $2t$ non-overlapping pairs of characters, where for some $a,b\in\Sigma$, each pair is either $ab$ or $ba$. Then, $w\overset{s}{\underset{(2\cdot t)}{\longrightarrow}}w'$ if and only if $w\underset{(2\cdot t)}{\longrightarrow}w'$ with respect to every  ordered ternary alphabet that has $a$ and $b$ as consecutive characters in it.
\end{remark}

The following theorem is analogous to the results established in the study of irreducible ($2\cdot t$) transformations in \cite{wT16}. In fact, due to Remark $\ref{RelateStrong2tAnd2t}$, one can verify that the former is a direct corollary of the latter. Thus the proof is not presented explicitly in this paper.
\begin{theorem}\label{Characterize2t}
Suppose $\Sigma$ is a ternary alphabet, $w,w'\in\Sigma^*$ and $t\ge 1$. Suppose $w'$ is obtained from $w$ by rewriting $2t$ non-overlapping pairs of characters, where for some $a,b\in\Sigma$, each pair is either $ab$ or $ba$. Let $(p_i,q_i)$ be the ordered pair defined in $(\ast)$ for every $1\le i\le 2t$. Then
\begin{equation*}
w\overset{s}{\underset{(2\cdot t)}{\longrightarrow}}w'\text{ if and only if }\sum\limits_{i=1}^{2t}p_i=\sum\limits_{i=1}^{2t}q_i=0.
\end{equation*}
Furthermore, if $w\overset{s}{\underset{(2\cdot t)}{\longrightarrow}}w'$, then the transformation is reducible if and only if there exists a non-empty $I\varsubsetneq\{1,2,\ldots,2t\}$ such that $\sum_{i\in I}p_i=\sum_{i\in I}q_i=0$.
\end{theorem}
The following result immediately holds from Theorem \ref{Characterize2t}.
\begin{remark}\label{ArbitraryPairing}
Suppose $\Sigma$ is a ternary alphabet, $w,w'\in\Sigma^*$, and $t\ge 1$ . Suppose $w'$ is obtained from $w$ by rewriting $2t$ non-overlapping pairs of characters, where for some $a,b\in\Sigma$, each pair is either $ab$ or $ba$. Assume $w\overset{s}{\underset{(2\cdot t)}{\longrightarrow}}w'$. Arbitrarily pair up $ab$ and $ba$ from those pairs in $w$ to form $t$ factors $\beta_k$ as in Theorem \ref{Strong2tTransformation}. Then, condition~(3) in Theorem \ref{Strong2tTransformation} is satisfied.
\end{remark}

The proof of the following theorem is analogous to the proof of Theorem~$4.9$ in \cite{wT16}. However, the proof is shown here in the setting of strong $M$\!-equivalence (more precisely, by using Theorem \ref{Characterize2t}) so that our exposition is self-contained.

\begin{theorem}
Suppose $\Sigma=\{a,b,c\}$. For every $t\ge 1$, there exist $w,w'\in\Sigma^*$ such that $w\overset{s}{\underset{(2\cdot t)}{\longrightarrow}}w'$ is irreducible.
\end{theorem}
\begin{proof}
Every strong $(2\cdot 1)$ transformation is irreducible. Fix an integer $t>1$. Let $w=(ab)^{t-1}c(ba)^tc^{t-1}ab$ and $w'=(ba)^{t-1}c(ab)^tc^{t-1}ba$. Then, the $2t$ pairs of characters being rewritten from left to right are $\underbrace{ab,\ldots ,ab}_{t-1 \text{ times}}\underbrace{ba,\ldots ,ba}_{t \text{ times}},ab$. The corresponding ordered pairs, $(p_i,q_i),1\le i\le 2t$ are 
\begin{equation*}
\underbrace{(-1,-t),\ldots ,(-1,-t)}_{t-1 \text{ times}}\underbrace{(1,t-1),\ldots ,(1,t-1)}_{t \text{ times}},(-1,0).
\end{equation*}
By Theorem \ref{Characterize2t}, $w\overset{s}{\underset{(2\cdot t)}{\longrightarrow}}w'$.

To prove that $w\overset{s}{\underset{(2\cdot t)}{\longrightarrow}}w'$ is irreducible, we argue by contradiction. Assume $w\overset{s}{\underset{(2\cdot t)}{\longrightarrow}}w'$ is reducible. Then, by Theorem \ref{Characterize2t}, there exists a non-empty $I\varsubsetneq\{1,2,\ldots ,2t\}$ such that $\sum_{i\in I}p_i=\sum_{i\in I}q_i=0$. Since $\sum_{i\in I}p_i=0$, it follows that $|\{i\in I\,|\,\mu_i=ab\}|=|\{i\in I\,|\,\mu_i=ba\}|=\frac{|I|}{2}$. Observe that if $2t\in I$, then $\sum_{i\in I}q_i=-(\frac{|I|}{2}-1)t+\frac{|I|}{2}(t-1)=t-\frac{|I|}{2}$. Otherwise, if $2t\not\in I$, then $\sum_{i\in I}q_i=-\frac{|I|}{2}t+\frac{|I|}{2}(t-1)=-\frac{|I|}{2}$. Since $0<|I|<2t$, it follows that $\sum_{i\in I}q_i\neq 0$ in either case, which is a contradiction. Thus, our conclusion holds.
\end{proof}

The following is a structural characterization of pairs of words that are obtained from one another by an irreducible strong ($2\cdot 2$) transformation.
\begin{theorem}\label{2.2Irreducible}
Suppose $\Sigma=\{a,b,c\}$ and $w,w'\in\Sigma^*$. Then $w\overset{s}{\underset{(2\cdot 2)}{\longrightarrow}}w'$ is irreducible iff  $w=w_1\boldsymbol{ab}w_2\boldsymbol{ba}w_3\boldsymbol{ba}w_4\boldsymbol{ab}w_5$ and 
$w'=w_1\boldsymbol{ba}w_2\boldsymbol{ab}w_3\boldsymbol{ab}w_4\boldsymbol{ba}w_5$ for some $w_i\in\Sigma^*,1\le i\le 5$ such that $|w_2|_c=|w_4|_c>0$.
\end{theorem}
\begin{proof}
The backward direction of the theorem is immediate from Theorem \ref{Strong2tTransformation} and Theorem~\ref{Characterize2t}, thus it remains to prove the forward direction. Suppose $w\overset{s}{\underset{(2\cdot 2)}{\longrightarrow}}w'$. Let $\beta_1$ and $\beta_2$ be the two factors in $w$ that are involved in the transformation. Then, by conditions (1) and (2) in Theorem~\ref{Strong2tTransformation}, for some $w_i\in\Sigma^*,1\le i\le 5$, either one of the following cases holds. Note that by Remark~\ref{ArbitraryPairing}, for each case, we can arbitrarily choose the associated set of factors $\{\beta_1,\beta_2\}$.

\begin{case}$w=w_1\boldsymbol{ab}w_2\boldsymbol{ba}w_3\boldsymbol{ab}w_4\boldsymbol{ba}w_5$ and 
$w'=w_1\boldsymbol{ba}w_2\boldsymbol{ab}w_3\boldsymbol{ba}w_4\boldsymbol{ab}w_5$.\\
Choose $\{\beta_1,\beta_2\}=\{abw_2ba, abw_4ba\}$. Then, by condition (3) in Theorem~\ref{Strong2tTransformation}, it follows that  $|w_2|_c=|w_4|_c=0$.
It remains to be seen that 
\begin{equation*}
w\overset{s}{\underset{(2\cdot 1)}{\longrightarrow}}w_1baw_2abw_3abw_4baw_5\overset{s}{\underset{(2\cdot 1)}{\longrightarrow}}w'. 
\end{equation*}
Therefore, $w$ and $w'$ cannot be in this form if $w\overset{s}{\underset{(2\cdot 2)}{\longrightarrow}}w'$ is irreducible.
\end{case}
\begin{case} $w=w_1\boldsymbol{ab}w_2\boldsymbol{ba}w_3\boldsymbol{ba}w_4\boldsymbol{ab}w_5$ and 
$w'=w_1\boldsymbol{ba}w_2\boldsymbol{ab}w_3\boldsymbol{ab}w_4\boldsymbol{ba}w_5$.\\
Choose $\{\beta_1,\beta_2\}=\{abw_2ba, baw_4ab\}$. Then, by condition (3) in Theorem \ref{Strong2tTransformation}, it follows that $|w_2|_c=|w_4|_c$.
Suppose $w\overset{s}{\underset{(2\cdot 2)}{\longrightarrow}}w'$ is irreducible. Assume that $|w_2|_c=|w_4|_c=0$. Then, by similar argument as in Case~1, $w\overset{s}{\underset{(2\cdot 2)}{\longrightarrow}}w'$ is reducible, which is a contradiction. Therefore, \mbox{$|w_2|_c=|w_4|_c>0$.}
\end{case}
\begin{case}$w=w_1\boldsymbol{ab}w_2\boldsymbol{ab}w_3\boldsymbol{ba}w_4\boldsymbol{ba}w_5$ and 
$w'=w_1\boldsymbol{ba}w_2\boldsymbol{ba}w_3\boldsymbol{ab}w_4\boldsymbol{ab}w_5$.\\
Choose $\{\beta_1,\beta_2\}=\{abw_2abw_3ba, abw_3baw_4ba\}$. Then, by condition (3) in Theorem \ref{Strong2tTransformation}, it follows that $|w_2|_c=|w_3|_c=|w_4|_c=0$. By similar argument as in Case~1, $w$ and $w'$ cannot be in this form if $w\overset{s}{\underset{(2\cdot 2)}{\longrightarrow}}w'$ is irreducible.
\end{case}
\end{proof}

\section{On the Characterization of Strong $M$\!-equivalence for the Ternary Alphabet}

It has been shown that the strong $(2\cdot t)$ and strong $(3\cdot t)$ transformations naturally preserve strong $M$\!-equivalence for the ternary alphabet. In fact, to develop a complete natural characterization of strongly $M$\!-equivalent ternary words, both rules are essential. The main question now is to what extent these transformations are able to provide such characterization.

To answer that question, we look into the notion of $\alpha\beta$-transformations and \textit{MSAE}-equivalence, both introduced in \cite{GT16a}.
\begin{definition}\label{AlphaBetaTransformation}
Suppose $\Sigma$ is an alphabet and $w,w'\in\Sigma^*$. Let $\alpha$ and $\beta$ be any two distinct characters of $\Sigma$. We say that $w$ transforms into $w'$ using an $\alpha\beta$-transformation, denoted by $w\overset{(\alpha\beta)}{\longrightarrow}w'$, if and only if $w=x\alpha\beta y\beta\alpha z$ and $w'=x\beta\alpha y\alpha\beta z$ for some $x,y,z\in\Sigma^*$.
\end{definition}
Observe that every strong $(2\cdot t)$ and strong $(3\cdot t)$ transformation is a disjoint combination of some $\alpha\beta$-transformations. Each $\alpha\beta$-transformation alone, however, does not always preserve strong $M$\!-equivalence.

Suppose $\Sigma=\{a,b,c\}$ and $w,w'\in\Sigma^*$. It was shown in \cite{GT16a} that whenever $w$ undergoes an $\alpha\beta$-transformation, it suffices to note only the changes in the number of occurrences of $abc, acb,$ and $bac$ as subwords in $w$. This consequently motivated the following definition.

\begin{definition}\label{MSAEcounter}
Suppose $\Sigma=\{a,b,c\}$ and $w,w'\in\Sigma^*$. Suppose $w\overset{(\alpha\beta)}{\longrightarrow}w'$ for any two distinct $\alpha ,\beta\in\Sigma$. Let $\delta_v=|w|_v-|w'|_v$ for every $v\in\{abc,acb,bac\}$. Then, we say that $w'$ is obtained from $w$ by an $\alpha\beta$-transformation with a \textit{counter} $\Omega=(\delta_{abc},\delta_{acb},\delta_{bac})$, denoted by $w\overset{\Omega}{\longrightarrow}w'$.
\end{definition}

\begin{definition}\label{DefineMSAE}
Suppose $\Sigma=\{a,b,c\}$ and $w,w'\in\Sigma^*$. We say that $w$ and $w'$ are \textit{almost strongly elementarily matrix equivalent} (\textit{MSAE}-equivalent), denoted by $w\equiv_{{\it MSAE}}w'$, if and only if there exist $w_0,w_1,\ldots , w_n\in\Sigma^*$ such that $w=w_0\overset{\Omega_1}{\longrightarrow}w_1\overset{\Omega_2}{\longrightarrow}w_2\overset{\Omega_3}{\longrightarrow}\cdots \overset{\Omega_n}{\longrightarrow}w_n=w'$ for some counters $\Omega_i,1\le i\le n,$ satisfying $\sum_{i=1}^{n}\Omega_i=(0,0,0)$.
\end{definition}

\begin{theorem}\label{CharacterizeStrM}\cite{GT16a}
Suppose $\Sigma=\{a,b,c\}$ and $w,w'\in\Sigma^*$. If $w\equiv_{{\it MSAE}}w'$, then $w\overset{s}{\equiv}_Mw'$.
\end{theorem}

Theorem \ref{CharacterizeStrM} established \textit{MSAE}-equivalence as a sound rewriting system in determining when two ternary words are strongly $M$\!-equivalent. We observe that the converse of Theorem \ref{CharacterizeStrM} is true if and only if the following is true.

\vspace{0.35em}\begin{quote}
\textit{Whenever two ternary words are strongly $M$\!-equivalent, they are obtainable from one another by finitely many $\alpha\beta$-transformations. }
\end{quote}\vspace{0.35em}

However, one can see that this property does not hold by looking at the pair of words $abcabcabcabcabcabc$ and $cabababcabccabccab$. These words are strongly $M$\!-equivalent with respect to $\{a,b,c\}$ but no $\alpha\beta$-transformation can be applied on the former word. Hence, to obtain a complete natural characterization of strong $M$\!-equivalence, we have to cleverly develop strong $M$\!-equivalence preserving rewriting rules that are not only combinations of $\alpha\beta$-transformations.

\section{Conclusion}
The quest to solve the strong injectivity problem for arbitrary \mbox{alphabet} \mbox{remains} open. However, for the ternary alphabet, we have shown that the strong $(2\cdot t)$ and strong $(3\cdot t)$ transformations constitute a sound expansion of the elementary rewriting rules. Furthermore, extending the notion of strong $(2\cdot t)$ transformation for arbitrary alphabet is a possible future work.

Finally, we conclude by pointing out an interesting problem, related directly to Theorem \ref{2.2Irreducible}.
\begin{question}\label{Ques}
Suppose $\Sigma$ is a ternary alphabet and $w,w'\in\Sigma^*$. If $w\overset{s}{\underset{(2\cdot 2)}{\longrightarrow}}w'$ is irreducible, can $w'$ be obtained from $w$ by finitely many strong $(2\cdot 1)$ transformations?
\end{question}
This problem is analogous to Conjecture~6.1 in \cite{wT16}, where it was surmised that if a word is obtained from another by an irreducible $(2\cdot 2)$ transformation, then the two words are not \textit{ME}-equivalent. Similarly, should the answer to Question~\ref{Ques} be no, then any two words obtained from one another by an irreducible strong $(2\cdot 2)$ transformation are not \textit{MSE}-equivalent.

\section{Acknowledgement}
We gratefully acknowledge the support for this research by a short term grant No. 304/PMATHS/6313077 of Universiti Sains Malaysia. Part of the work has been presented orally in the 4th International Conference on Mathematical Sciences (ICMS4) held on 15-17 Nov 2016 at Putrajaya, Malaysia.


\end{document}